\documentclass[a4paper,11pt]{article}

\usepackage{braket}
\usepackage{listings}
\usepackage{amsmath}
\usepackage{amsfonts}
\usepackage{amsthm}
\usepackage{amssymb}
\usepackage{ascmac}
\usepackage{fancybox}

\usepackage{newtxtext}

\usepackage{titlesec}
\usepackage{here}
\usepackage{bm}
\usepackage{caption}
\usepackage{enumerate}
\usepackage{pdfpages}

\usepackage{siunitx}
\usepackage{url}

\usepackage{a4wide}

\usepackage{graphicx}

\usepackage[linktocpage, colorlinks=true]{hyperref}
\usepackage[capitalise]{cleveref}

\theoremstyle{definition}
\newtheorem{theorem}{Theorem}[section]
\newtheorem{proposition}[theorem]{Proposition}
\newtheorem{corollary}[theorem]{Corollary}
\newtheorem{lemma}[theorem]{Lemma}
\newtheorem{definition}[theorem]{Definition}
\newtheorem{remark}[theorem]{Remark}
\newtheorem{example}[theorem]{Example}

\makeatletter
\@addtoreset{figure}{section}
\@addtoreset{table}{section}
\@addtoreset{equation}{section}
\makeatother

\newcommand{\RR}{\mathbb{R}}

\newcommand{\ZZ}{\mathbb{Z}}
\newcommand{\NN}{\mathbb{N}}

\newcommand{\Di}{\mathcal{E}}

\newcommand{\calF}{\mathcal{F}}

\usepackage{docmute}

\begin{document}

\title{Characterization of Subordinate Symmetric Markov Processes}
\author{Ryuto Kushida}
\date{}

\maketitle

\begin{abstract}
  In this paper, we consider subordinate symmetric Markov processes 
  which correspond to non-killing Dirichlet forms
  enjoying heat kernel estimates on a metric measure space with the  volume doubling property.
  We obtain estimates of the jump kernel of the subordinate process and 
  establish equivalent conditions for the jump kernel following Liu-Murugan \cite{LM}.
  In particular, we clarify the scale of the jump kernel,
  which is different from the diffusion type.
  This result is appliable to non-subordinate processes by the transferring method, which uses
  stability of Dirichlet forms.
\end{abstract}


\section{Introduction}

In recent years, 
there have been significant amount of researches on
jump processes.
Since generators of Markov processes with discontinuous paths are non-local operators, 
they are studied both from the areas of probability and partial differential equations.
One of the useful approaches to studying jump processes 
is the time change method via subordinators, called subordination.

A typical example of subordination
is the $\alpha$-stable process on $\RR^d$ 
which is obtained from Brownian motion by a specific subordinator.
Concerning applications, 
subordinate Brownian motions on $\RR^d$ or on a bounded domain have been investigated,  
and significant amount of properties such as heat kernel estimates and Harnack inequalities
have been revealed
\cite{KM2012},
\cite{KM2014},
\cite{KSV},
\cite{KSV2},
\cite{M}.
Subordination is applicable 
not only to Brownian motion but also 
to Markov processes on metric measure spaces by using the theory of Dirichlet forms.
To be precise, let $(M,d,\mu)$ be a metric measure space, that is, $(M,d)$ is a locally compact separable metric space and $\mu$ is a full support Radon measure on $M$.
\^Okura \cite{Okura} established a formula of Dirichlet forms for subordinate Markov processes.
Research on jump processes especially on fractals 
(for example, Sierpinski gasket in Figure \ref{fig:SierpinskiGasket} and Sierpinski carpet in Figure \ref{fig:SierpinskiCarpet})
has been developed
since the beginning of this century, 
see for instance Chen-Kumagai \cite{CK2003}.
A typical example is ``$\alpha$-stable process''
whose  Dirichlet form is given by
\begin{equation}
    \Di(f,f) = \int_{M \times M}(f(x) - f(y))^2 J(x,y)\mu(dx)\mu(dy), \quad    
    J(x,y) = \frac{C}{V(x,d(x,y))d(x,y)^\alpha},
\end{equation}
where $V(x,r) := \mu(B(x,r))$ and $B(x,r) := \{y \in M;d(x,y) < r\}$.
When $M = \RR^d$, this is a classical $\alpha$-stable process.
More generally, stable-like cases and mixed type cases are considered in many literatures
(see for instance \cite{BBCK}, \cite{BSS}, \cite{CK2003},  \cite{CKW2021}, \cite{K}, \cite{MS}).
Here stable-like means that the jump kernel $J(x,y)$ is within the range of constants times that of $\alpha$-stable process :
\begin{equation}
  \frac{C_1}{V(x,d(x,y))d(x,y)^{\alpha}} \le J(x,y) \le
  \frac{C_2}{V(x,d(x,y))d(x,y)^{\alpha}},
\end{equation}
and mixed type means that the jump kernel $J(x,y)$ is bounded by two scales $\alpha_1,\alpha_2$ :
\begin{equation}
  \frac{C_1}{V(x,d(x,y))d(x,y)^{\alpha_1}} \le J(x,y) \le
  \frac{C_2}{V(x,d(x,y))d(x,y)^{\alpha_2}}.
\end{equation}
When $\alpha < 2$, this Dirichlet form is regular and the corresponding stochastic process exists on the metric measure space in general.
Moreover, in recent years, the case $\alpha_2 \ge 2$ is considered.
Due to the existence of the sub-diffusion on fractals (see \cref{ex:subDiffusion}), 
these processes exist and Chen-Kumagai-Wang \cite{CKW2021} established stability of the heat kernel estimates under suitable perturbation. 
On the other hand, Chen-Kumagai-Wang \cite{CKW2022} and Bae-Kang-Kim-Lee \cite{BKKL2019}, \cite{BKKL} independently obtained 
the heat kernel estimates of a jump process having such general polynomial growth type.
Quite recently, Liu-Murugan \cite{LM} obtain the equivalence for jump kernels
deduced by subordination of diffusion processes with the so-called sub-Gaussian heat kernel estimates on metric measure spaces.

Based on these work,
we consider a subordinate process whose original process have jumps
unlike the setting of \cite{BKKL}
and obtain an equivalence for the jump kernel (Theorem \ref{thm:diffusion+jump}).
In order to describe our main result,
we first define heat kernel estimates of a diffusion process with jumps (diffusion+jump in short).
Let $\phi_c,\phi_j$ be scale functions that enjoys the doubling property (see Definition \ref{def:scaleFunction}).
$\phi_c$ represents the scale of time and space on the diffusion part
and $\phi_j$ represents that of the jump part. 
Set $\phi = \phi_c \wedge \phi_j$.
\begin{definition}(Diffusion+jump type heat kernel estimates \cite[Definition 1.11]{CKW2020})
  \label{def:diffusion+jumpHKE}

  For a regular Dirichlet form $(\Di,\calF)$ of the diffusion+jump type, 
  we say HK$_{-}$($\phi_c,\phi_j$) holds if there exists heat kernel $p(t,x,y)$ of $(\Di,\calF)$
  and 
  there exist a properly exceptional set $\mathcal{N}$ and 
  constants $c_1 ,c_2,c_3,c_4 > 0$ such that
  for all $t > 0$ and $x,y \in M \backslash \mathcal{N}$
  \begin{equation}
    \begin{split}
        &c_1\left(\frac{1}{V(x,\phi^{-1}(t))} 1_{\{\phi(c_2 r) \le t\}} + \frac{t}{V(x,r)\phi_j(r)} 1_{\{\phi(c_2 r) > t\}}\right)\\
    &\le p(t,x,y)\\
    &\le c_3\left(\frac{1}{V(x,\phi^{-1}(t))} \wedge (p^{(c)}(c_4t,x,y) + p^{(j)}(t,x,y))\right).
    \end{split}
    \label{eq:diffusion+jumpHKE}
  \end{equation}
Here $p^{(c)}(t,x,y)$ defined in (\ref{eq:diffusionHKE}) appears as the sub-Gaussian heat kernel and 
  $p^{(j)}(t,x,y)$ defined in (\ref{eq:JumpHK}) appears as the stable-like heat kernel.
See Example \ref{ex:diffusion+jump} for concrete examples of processes that enjoy (\ref{eq:diffusion+jumpHKE}).
\end{definition}

Next we introduce estimates of jump kernels.
For a diffusion+jump process satisfying the heat kernel estimates HK$_-(\phi_c,\phi_j)$, 
and for a scale function $\psi$,
the jump kernel of the subordinate process subordinate by $\overline{\phi}$
is obtained as follows.
  \begin{equation}
    J(x,y) \simeq \frac{1}{V(x,r)}\left(\frac{1}{\psi(r)} +\overline{\phi}(\phi_j(r)^{-1})\right), \quad r > 0,
      \label{eq:jumpKernelEstimateOfdiffusion+jump}
  \end{equation}
where $\overline{\phi}$ is the Bernstein function defined by 
\begin{align}
  \label{eq:Bernstein function}
  \overline{\phi}(\lambda) &:= \int_0^{\infty} (1 - e^{-\lambda t}) \frac{dt}{t\psi(\phi^{-1}(t))}, \quad \lambda > 0.
\end{align}
The following theorem is the main result in this article.
\begin{theorem}
  \label{thm:diffusion+jump}
  Let $(\Di,\calF)$ be a diffusion+jump type regular Dirichlet form on $L^2(M;\mu)$ that 
  satisfies the heat kernel estimates HK$_{-}(\phi_c,\phi_j)$.
  $X$ denotes the $\mu$-symmetric Hunt process corresponding to $(\Di,\calF)$.
  Given a scale function $\psi$, the following are equivalent. 
  \begin{enumerate}[(a)]
    \item There exists a pure jump type regular Dirichlet form $(\Di^j,\calF^j)$ whose jump kernel satisfies (\ref{eq:jumpKernelEstimateOfdiffusion+jump}).
    \item There exists a subordinator $S_t$\ such that the jump kernel of the subordinated process $X_{S_t}$ satisfies  (\ref{eq:jumpKernelEstimateOfdiffusion+jump}).
    \item The scale function $\psi$ satisfies
    \begin{equation}
      \int_0^1 \frac{\phi(s)}{s\psi(s)}ds < \infty.
    \end{equation}
  \end{enumerate}
\end{theorem}
This result is a generalization of Liu-Murugan \cite[Theorem 2.3]{LM} to diffusion+jumps.
\begin{remark}
  \begin{enumerate}[(i)]
    \item When the original process is the pure jump type 
        and enjoys the stable-like heat kernel estimates (Definition \ref{def:JumpHKE}),
        we also obtain the counterpart of 
        Theorem \ref{thm:diffusion+jump} and 
        the jump kernel estimates (\ref{eq:jumpKernelEstimateOfdiffusion+jump})
        ; see Corollary \ref{cor:jump}.
    \item A novelty of this paper is clarifying a scale function $\overline{\phi}(\phi_j(r)^{-1})$
        in the general setting of diffusion+jumps. 
        This scale $\overline{\phi}(\phi_j(r)^{-1})$ is not comparable to $1/\psi(r)$ 
        which appeared in \cite{BKKL}, \cite{LM} (see Example \ref{ex:comparable}).
    \item The estimates (\ref{eq:jumpKernelEstimateOfdiffusion+jump}) based on Proposition \ref{prop:diffusion+jumpType} is new even for $\RR^n$.
      \end{enumerate}
\end{remark}
Below, we list up some notation used in this paper.
\begin{itemize}
  \item $f \simeq g$ for real valued functions $f,g$, if there exists a constant $C > 0$ such that $C^{-1}f \le g \le C f$.
  \item $a \wedge b := \min(a,b),\ a \vee b:= \max(a,b)$.
  \item $C_c(M)$ is a set of continuous functions on $M$ whose supports are compact.
\end{itemize}
We note that in this paper constants $C$ are changed from line to line.

The rest of this paper is organized as follows.
Section 2 is preliminaries for heat kernel estimates and subordinators.
In Section 3, we first calculate the jump kernel of the subordinate process.
Then we give the proof of our main theorem.

\begin{figure}[H]
  \centering
  \begin{minipage}[b]{0.48\columnwidth}
    \centering
    \includegraphics[width=0.9\columnwidth]{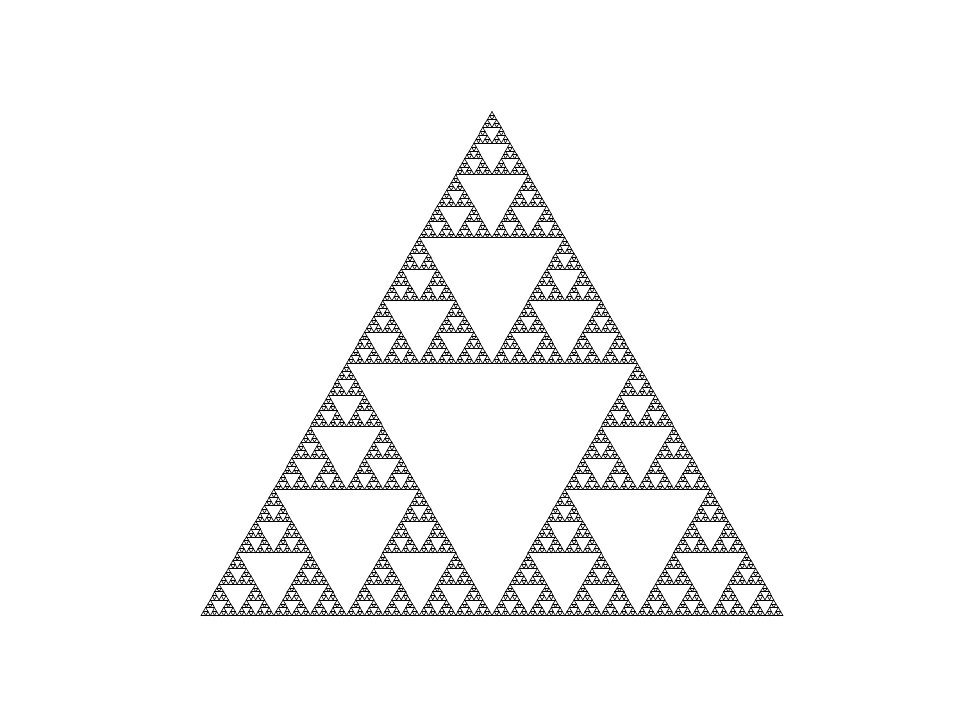}
    \caption{Sierpinski gasket}
    \label{fig:SierpinskiGasket}
  \end{minipage}
  \begin{minipage}[b]{0.48\columnwidth}
    \centering
    \includegraphics[width=0.9\columnwidth]{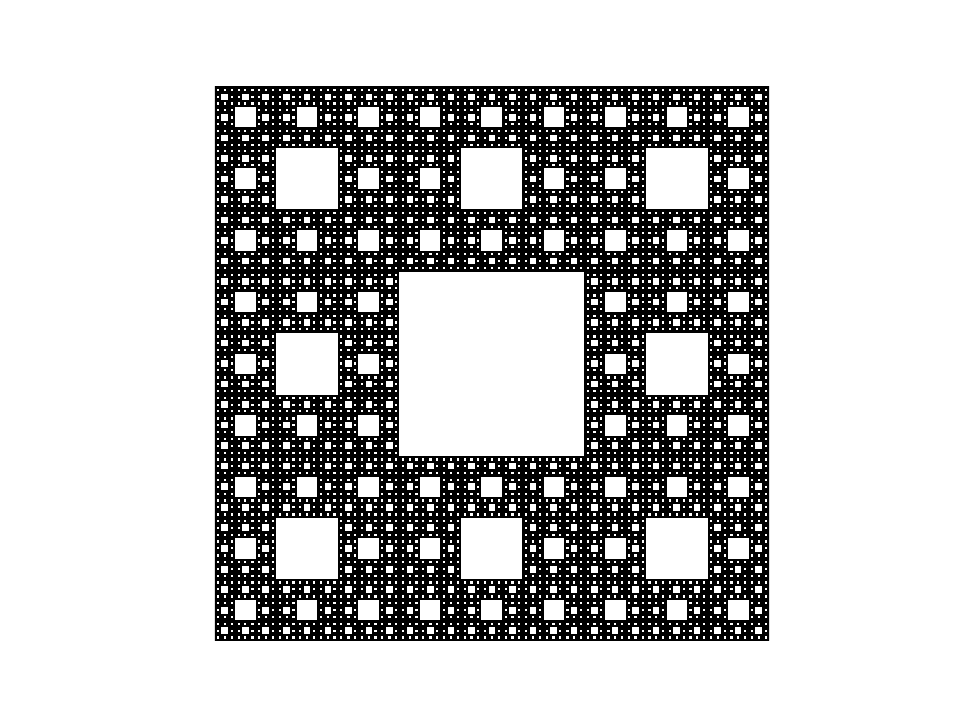}
    \caption{Sierpinski carpet}
    \label{fig:SierpinskiCarpet}
  \end{minipage}
\end{figure}

\section{Preliminaries}

\begin{definition}
  \begin{enumerate}[(i)]
    \item  We say that $(M,d,\mu)$ satisfies the volume doubling property (VD)
  if there exist constants $C \ge 1,d_2 > 0$ such that for all $x \in M$ and all $0 < r \le R$,
  \begin{equation}
    \frac{V(x,R)}{V(x,r)} \le C \left(\frac{R}{r}\right)^{d_2}.
  \end{equation}
  \item   We say that $(M,d,\mu)$ satisfies the reverse volume doubling property (RVD)
  if there exist constants $C > 0,d_1 > 0$ such that for all $x \in M$ and all $0 < r \le R$,
  \begin{equation}
   C \left(\frac{R}{r}\right)^{d_1} \le \frac{V(x,R)}{V(x,r)}.
  \end{equation}
  \end{enumerate}
\end{definition}
We always assume $(M,d,\mu)$ satisfies (VD) and do not need to assume (RVD).

\begin{definition}[Scale function]
  \label{def:scaleFunction}
  For constants $0 < \alpha_1 \le \alpha_2$, 
  we say a continuous function $\psi : [0,\infty) \to [0,\infty)$ satisfies $LU(\alpha_1,\alpha_2)$
  if $\psi$ is strictly increasing, $\psi(0) = 0$, $\psi(1) = 1$  and 
  there exists constants $C \ge 1$ such that
  \begin{align}
    C^{-1} \left(\frac{R}{r}\right)^{\alpha_1} \le \frac{\psi(R)}{\psi(r)} \le C \left(\frac{R}{r}\right)^{\alpha_2}, \qquad 0 < r \le R.
  \end{align}
  We call such function $\psi$ a scale function.
\end{definition}

\subsection{Dirichlet forms and heat kernel estimates}

We use notations and basic results of Dirichlet forms in Fukushima-Oshima-Takeda \cite{FOT}.
Let $\calF$ be a dense linear subspace of $L^2(M;\mu)$.
A Dirichlet form $(\Di,\calF)$ on $L^2(M;\mu)$ is a symmetric form $\Di: \calF \times \calF \to \RR$
satisfying the closedness property ($\calF$ is a Hilbert space under the norm $\Di_1(\cdot,\cdot) = \Di(\cdot,\cdot) + (\cdot,\cdot)_{L^2(M;\mu)}$)
and the unit contraction property 
($u \in \calF$ implies $v := (0 \vee u) \wedge 1 \in \calF$ and $\Di(v,v) \le \Di(u,u)$).
We say a Dirichlet form $(\Di,\calF)$ is regular
if $C_c(M) \cap \calF$ is dense in $C_c(M)$ with respect to the uniform norm
and dense in $\calF$ with respect to $\Di_1$.

From the Beurling--Deny decomposition (\cite[Theorem 3.2.1]{FOT}), 
a regular Dirichlet form  $(\Di,\calF)$ is uniquely expressed by
\begin{equation}
  \Di(u,u) = \Di^{(c)}(u,u) + \int_{M\times M \backslash \mathrm{diag}} (u(x) - u(y))^2 J(dx,dy) + \int_M u^2 dk,
\end{equation}
where $\Di^{(c)}$ is a symmetric form with the strongly local property
($\Di^{(c)}(u,v) = 0$ for $u,v \in \calF$ provided that $v$ is constant on a neighbourhood of $\mathrm{supp}\ u$)
and $J$\ is a symmetric Radon measure on $M \times M\backslash \mathrm{diag}$ (called the jumping measure)
and $k$ is a Radon measure (called the killing measure).
For a given regular Dirichlet form $(\Di,\calF)$ on $L^2(M;\mu)$,
by \cite[Theorem 7.2.1]{FOT}, 
there exists a $\mu$-symmetric Hunt process $X$ on $M$.
This process is unique up to the equivalence.

We are concerned with a regular Dirichlet form $(\Di,\calF)$ that is either a pure-jump type : 
\begin{align}
  \Di(f,f) &= \int_{M\times M \backslash \mathrm{diag}} (f(x) - f(y))^2 J(dx,dy),
\end{align}
or a diffusion+jump (i.e. no-killing inside) type : 
\begin{align}
  \Di(f,f) &= \Di^{(c)}(f,f) + \int_{M\times M \backslash \mathrm{diag}} (f(x) - f(y))^2 J(dx,dy),
\end{align}
Now, we will introduce heat kernel estimates.
For the following definition, we refer the reader to \cite[Theorem 1.11]{CKW2021}.
\begin{definition}[Jump type heat kernel estimates]
  \label{def:JumpHKE}
  Let $\phi_j$ be a scale function.
  For a  regular Dirichlet form $(\Di,\calF)$ of pure-jump type, 
  we say that HK($\phi_j$) holds if there exists a heat kernel $p(t,x,y)$ of $(\Di,\calF)$
  and 
  there exist a properly exceptional set $\mathcal{N}$ and 
  constants $c_1 ,c_2 > 0$ such that
  for all $t > 0$ and $x,y \in M \backslash \mathcal{N}$,
  \begin{align}
    c_1 p^{(j)}(t,x,y) \le p(t,x,y) \le  c_2p^{(j)}(t,x,y).
    \label{eq:JumpHKE}
  \end{align}
  Here, $p^{(j)}(t,x,y)$ is defined by
  \begin{equation}
    p^{(j)}(t,x,y) := \frac{1}{V(x,\phi_j^{-1}(t))} \wedge \frac{t}{V(x,d(x,y))\phi_j(d(x,y))}, \quad t > 0,x,y \in M.
    \label{eq:JumpHK}
  \end{equation}
\end{definition}

We will give examples of processes enjoying HK$(\phi_j)$.
\begin{example}
  \begin{enumerate}[(i)]
    \item Let $M = \RR^d$, $\mu$ be the Lebesgue measure and
    $Y_t$ be a symmetric $\alpha$-stable process where $\alpha \in (0,2)$ and 
    we put $\phi_j(r) = r^\alpha$.
    Then the Dirichlet form corresponding to $Y_t$ is known\cite[Example 1.4.1]{FOT} as 
    \begin{equation}
      \left\{
      \begin{aligned}
        \calF &:= \left\{u \in L^2(\RR^d); \int_{\RR^d}\int_{\RR^d} \frac{(u(x) - u(y))^2}{|x-y|^{d+\alpha}}dxdy < \infty \right\},\\
        \Di(u,v) &:= C_{d,\alpha} \int_{\RR^d}\int_{\RR^d} \frac{(u(x) - u(y))(v(x) - v(y))}{|x-y|^{d+\alpha}}dxdy,
      \end{aligned}
      \right.
    \end{equation}
    and $(\Di,\calF)$ satisfies HK$(\phi_j)$.
    \item Another examples are stable-like process
     on on $d$-sets (see \cite{BSS}, \cite{CK2003}, \cite{CK2008}).
     Let $G$ be the Sierpinski gasket whose one vertex is the origin.
     We define unbounded Sierpinski gasket $F := \bigcup_{m \in \NN} 2^m G$
     and $\mu$ be its Hausdorff measure (see Figure \ref{fig:unboundedSierpinskiGasket}).
     In \cite[p.3]{CKW2021}, put $d = \log (n+1)/\log 2, \alpha \in (0,2)$ and 
     \begin{equation}
      \left\{
      \begin{aligned}
        \calF &:= \left\{u \in L^2(F;\mu); \int_{F}\int_{F} \frac{(u(x) - u(y))^2}{|x-y|^{d+\alpha}}d\mu d\mu < \infty \right\},\\
        \Di(u,v) &:= \int_{F}\int_{F} \frac{(u(x) - u(y))(v(x) - v(y))}{|x-y|^{d+\alpha}}d\mu d\mu,
      \end{aligned}
      \right.
    \end{equation}
    then $(\Di,\calF)$ satisfies HK$(\phi_j)$ for $\phi_j(r) = r^\alpha$.
  \end{enumerate}
\end{example}

\begin{figure}[H]
  \centering
  \begin{minipage}[b]{0.48\columnwidth}
    \centering
    \includegraphics[width=0.9\columnwidth]{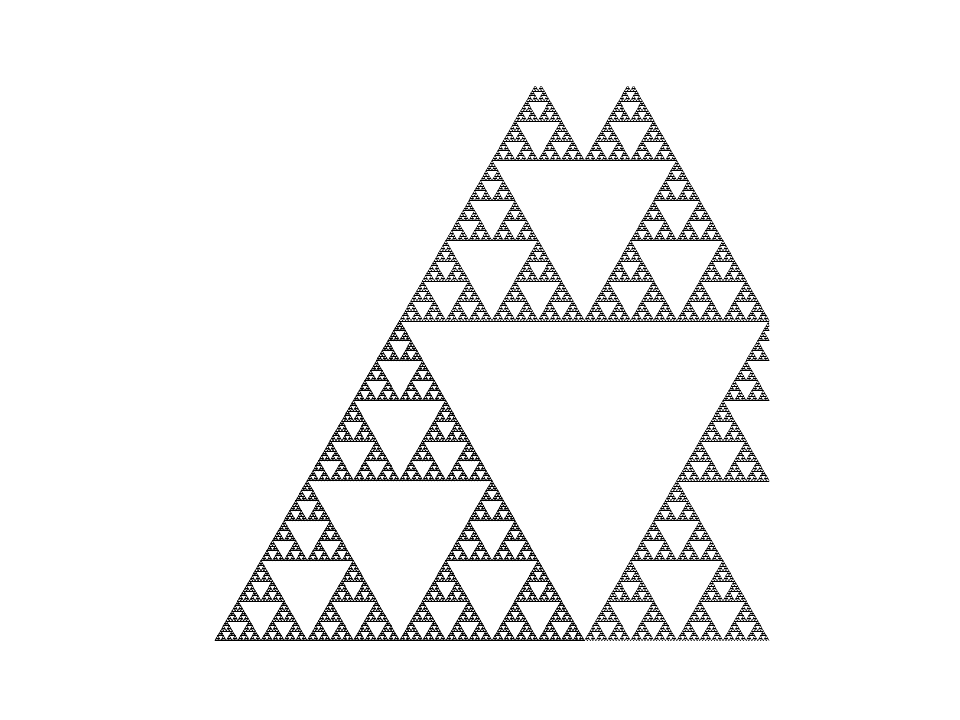}
    \caption{Unbounded Sierpinski gasket}
    \label{fig:unboundedSierpinskiGasket}
  \end{minipage}
  \begin{minipage}[b]{0.48\columnwidth}
    \centering
    \includegraphics[width=0.68\columnwidth]{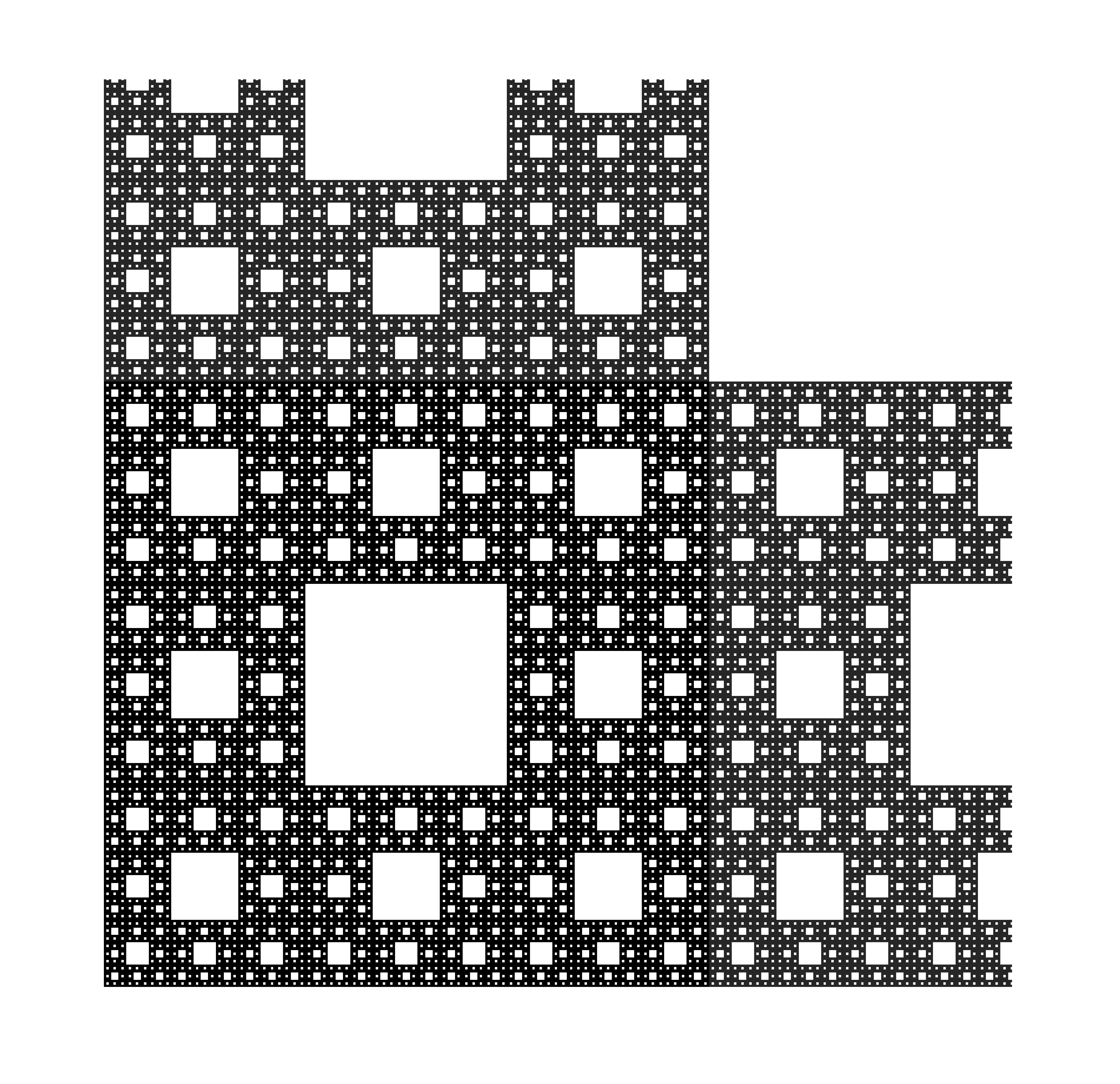}
    \caption{Unbounded Sierpinski carpet}
    \label{fig:unboundedSierpinskiCarpet}
  \end{minipage}
\end{figure}

\begin{remark}
  \begin{enumerate}[(i)]
    \item By \cite[Remark 1.12]{CKW2021}, the heat kernel $p(t,x,y)$ is 
    H\"older continuous with respect to $(x,y)$ for all $t > 0$ and (\ref{eq:JumpHKE}) holds for all $x,y \in M$ instead of $x,y \in M\backslash \mathcal{N}$ under (RVD).
     \item Equivalent conditions for $\mathrm{HK}(\phi_j)$ using function inequalities are known, see for insatnce \cite[Theorem 1.13]{CKW2021}.
  \end{enumerate}
 \end{remark}

Next, we introduce diffusion+jump type heat kernel estimates
defined in Definition \ref{def:diffusion+jumpHKE}.
We refer the reader to \cite[Definition 1.11]{CKW2020}.
Let $\phi_c,\phi_j$ be scale functions 
satisfying $LU(\beta_{1,c},\beta_{2,c})$, $LU(\beta_{1,j},\beta_{2,j})$
with $1 < \beta_{1,c} \le \beta_{2,c}$, $0 < \beta_{1,j} \le \beta_{2,j} < \infty$.
Moreover, we assume that $\phi_c$ and $\phi_j$ satisfy
\begin{equation}
  \phi_c(r) \le \phi_j(r), \quad r \in (0,1] \quad \mathrm{and}
   \quad \phi_c(r) \ge \phi_j(r), \quad r \in (1,\infty).
   \label{eq:relationOfScaleFunction}
\end{equation}
Note that $\phi_j$ here may be different from that of the jump type case.
We put $\phi(r) := \phi_c(r) \wedge \phi_j(r)$. 
$\phi$ is also a scale function satisfying $LU(\beta_{1,c} \wedge \beta_{1,j},\beta_{2,c} \vee \beta_{2,j})$.
Define 
\begin{equation}
  p^{(c)}(t,x,y) := \frac{1}{V(x,\phi_c^{-1}(t))}\exp \left(-\sup_{s > 0}\left\{\frac{d(x,y)}{s} - \frac{t}{\phi_c(s)}\right\}\right),\quad t > 0,x,y \in M.
  \label{eq:diffusionHKE}
\end{equation}
The function $p^{(c)}(t,x,y)$ appears in the study of diffusion processes on fractals (see \cite{B}\cite{BP}\cite{GT}).
Typical examples of heat kernels bounded by $p^{(c)}(t,x,y)$ are 
the Gaussian kernel 
and sub-Gaussian heat kernels that emerged in anomalous diffusion on fractals.
We give some concrete examples of sub-Gaussian heat kernels.
\begin{example} 
  \label{ex:subDiffusion}
  \begin{enumerate}[(i)]
    \item Brownian motion on the unbounded Sierpinski gasket in $\RR^n$ (Figure \ref{fig:unboundedSierpinskiGasket}) \cite{BP} :
    Let $d = \log(n+1)/\log2 , \beta = \log(n+3)/\log2, \phi_c(r) = r^\beta$.
    Then $V(x,r) \simeq r^d$ and there exists a diffusion (Brownian motion) whose heat kernel enjoys
    sub-Gaussian heat kernel estimates with 
    \begin{equation}
      p^{(c)}(t,x,y) \simeq \frac{1}{t^{d/\beta}}\exp \left(-\left(\frac{d^\beta(x,y)}{t}\right)^{1/(\beta - 1)}\right).
    \end{equation}
    \item Brownian motion on the unbounded Sierpinski carpet in $\RR^n$ (Figure \ref{fig:unboundedSierpinskiCarpet}) \cite{BB1989} \cite{BB} :
    Let $d = \log(3^n -1)/\log3$. 
    Then $V(x,r) \simeq r^d$ and there exist $d_w^c > 2$ and a diffusion (Brownian motion) whose heat kernel enjoys
    sub-Gaussian heat kernel estimates with 
    \begin{equation}
      p^{(c)}(t,x,y) \simeq \frac{1}{t^{d/d_w^c}}\exp \left(-\left(\frac{d^{d_w^c}(x,y)}{t}\right)^{1/(d_w^c - 1)}\right).
    \end{equation}
    \item On the tiling of 2-dimensional Sierpinski carpets (Figure \ref{fig:tilingCarpet}) :
    Let $K$ be the 2-dimensional compact Sierpinski carpet with vertices $(0,0), (0,1), (1,0), (1,1)$
    and set $K_{ij} = K + (i,j)$ for $i,j \in \ZZ$.
    We call $\bigcup_{i,j \in \ZZ} K_{ij}$ the tiling of 2-dimensional Sierpinski carpets.
    Let $\phi_c(r) = r^{d_w^c}1_{\{r \le 1\}} + r^{2}1_{\{r > 1\}}$.
    Since this space is roughly isometric to $\RR^2$ in the sence of \cite{BBK},
    there is a diffusion whose heat kernel satisfies
    sub-Gaussian heat kernel estimates.
  \end{enumerate}
\end{example}
\begin{figure}[H]
  \centering
  \begin{minipage}[b]{0.48\columnwidth}
    \centering
    \includegraphics[width=0.9\columnwidth]{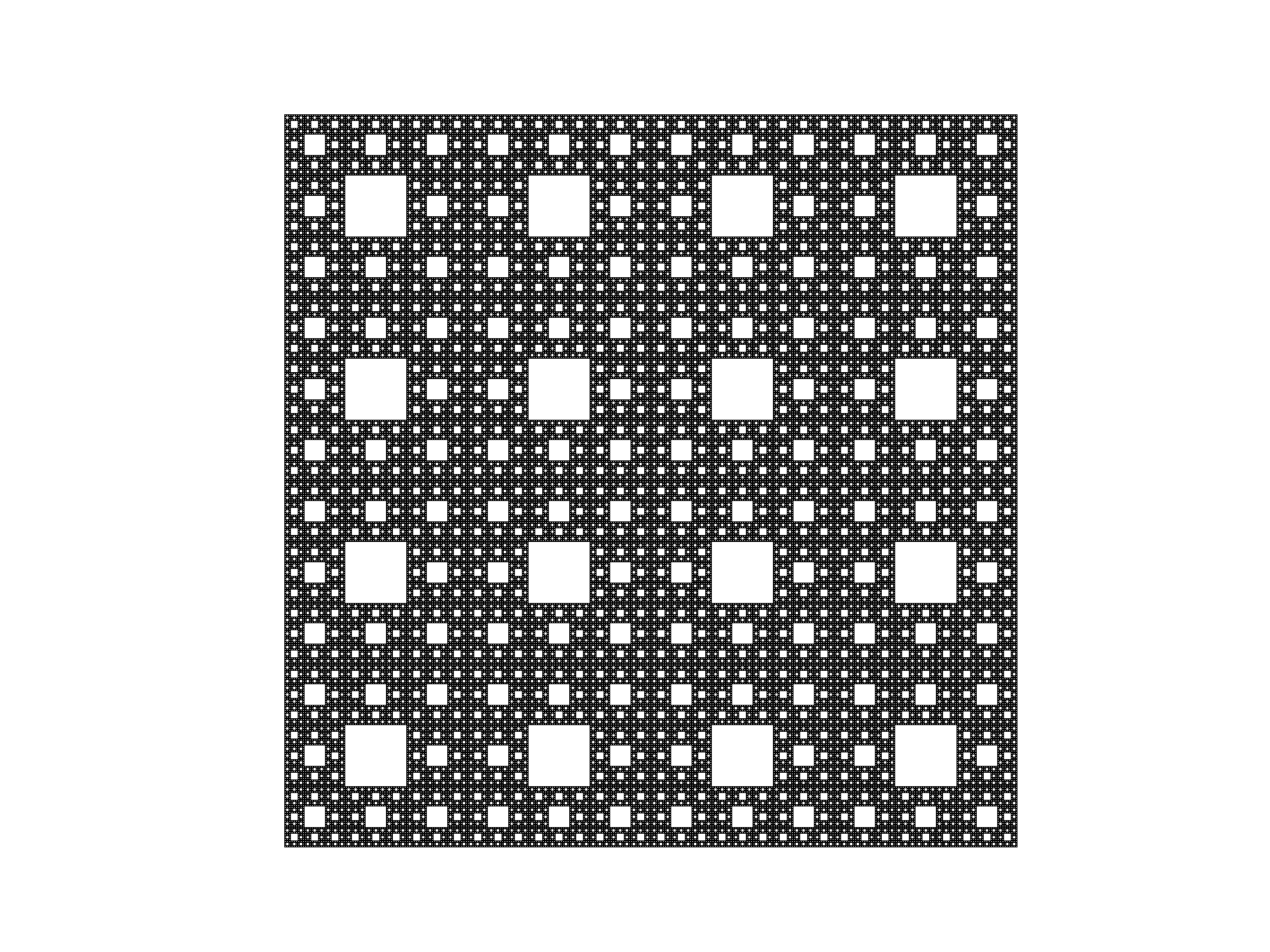}
    \caption{Tiling of Sierpinski carpets}
    \label{fig:tilingCarpet}
  \end{minipage}
\end{figure}

We give examples of processes enjoying HK$_-(\phi_c,\phi_j)$.
\begin{example}
  \label{ex:diffusion+jump}
  \begin{enumerate}[(i)]
    \item Let $M = \RR^d$ and $\mu$ be the Lebesgue measure.
    We take $B_t$ be Brownian motion and 
    $Y_t$ be an independent symmetric $\alpha$-stable process,
     where $\alpha \in (0,2)$.
    Put $\phi_c(r) := r^2$ and $\phi_j(r) := r^\alpha$.
    Then the sum of Brownian motion and the stable process $X_t := B_t + Y_t$ satisfies HK$_-(\phi_c,\phi_j)$ \cite[Theorem 2.13]{SV}.
    It is known\cite[Proposition 1.1]{CK2010} that the corresponding Dirichlet form $(\Di,\calF)$ is 
    \begin{equation}
      \left\{
      \begin{aligned}
        \calF &:= W^{1,2}(\RR^d)\\
        \Di(u,v) &:= \frac{1}{2} \int_{\RR^d} \nabla u \cdot \nabla v\ dx + C_{d,\alpha}\int_{\RR^d}\int_{\RR^d} \frac{(u(x) - u(y))(v(x) - v(y))}{|x-y|^{d+\alpha}}dxdy.
      \end{aligned}
      \right.
    \end{equation}
    
    \item Similarly, it is known \cite[Theorem 1.4]{CK2010} that the Dirichlet form defined by
    \begin{equation}
      \left\{
      \begin{aligned}
        \calF &:= W^{1,2}(\RR^d)\\
        \Di(u,v) &:= \int_{\RR^d} \nabla u \cdot A(x)\nabla v\ dx + \int_{\RR^d}\int_{\RR^d} (u(x) - u(y))(v(x) - v(y)) J(x,y) dxdy.
      \end{aligned}
      \right.
    \end{equation} satisfies HK$_-(\phi_c,\phi_j)$,
    where $A(x)$ is uniform elliptic 
    (i.e. $A(x)$ is symmetric, measurable and $C^{-1} |\xi|^2 \le \braket{\xi,A(x)\xi} \le C |\xi|^2$, for all $x$)
    and $J$ is a mixed type jump kernel in the sence of \cite[p.554]{CK2010}.
    An example of a mixed type jump kernel is 
    \begin{equation}
      J(x,y) := \int_{\alpha_1}^{\alpha_2} \frac{1}{|x-y|^{d+\alpha}}\nu(d\alpha),
    \end{equation}
    where $0 < \alpha_1 \le \alpha_2 < 2$ and $\nu$ be a probability measure on $[\alpha_1,\alpha_2]$.
    \item Let $(\Di^{(c)},\calF)$ be a Dirichlet form corresponding to Brownian motion
    on the $n$-dimensional unbounded Sierpinski gasket $M$ in $\RR^n$ (see \cref{ex:subDiffusion}).
    Put \begin{equation}
      \Di(u,v) := \Di^{(c)}(u,v) + \int_{M\times M}(u(x) - u(y))(v(x) - v(y)) \frac{1}{d(x,y)^{d+\alpha}}\mu(dx)\mu(dy),
    \end{equation}
    where $d = \log(n+1)/\log 2$, $\beta = \log(n+3)/\log 2$ and $\alpha \in (0,\beta)$.
    Then, $(\Di,\calF)$ is a regular Dirichlet form satisfying HK$_-(\phi_c,\phi_j)$ \cite[Example 7.2]{CKW2020}
    for $\phi_c(r) = r^\beta, \phi_j(r) = r^\alpha$.
  \end{enumerate}
\end{example}

\begin{remark}
  \begin{enumerate}[(i)]
    \item Equivalent conditions for $\mathrm{HK}_-(\phi_c,\phi_j)$ using function inequalities are known \cite[Theorem 1.13]{CKW2020}.
    \item In particular, if $(M,d,\mu)$ is connected and enjoys (RVD) and the chain condition,
      then HK$_-(\phi_c,\phi_j)$ leads HK$(\phi_c,\phi_j)$, namely 
      \begin{equation}
        \begin{split}
        &c_5\left(\frac{1}{V(x,\phi^{-1}(t))} \wedge (p^{(c)}(c_6t,x,y) + p^{(j)}(t,x,y))\right)\\
        &\le p(t,x,y)\\
        &\le c_3\left(\frac{1}{V(x,\phi^{-1}(t))} \wedge (p^{(c)}(c_4t,x,y) + p^{(j)}(t,x,y))\right).
        \end{split}
      \end{equation}
    \item There are pure jump Dirichlet forms that satisfy $\mathrm{HK}_{-}(\phi_c,\phi_j)$,
    see \cite[Definition 2.8]{BKKL} and \cite[Definition 1.10]{CKW2022}.
      Proposition \ref{prop:diffusion+jumpType} in the next section also holds for such processes.
\end{enumerate}
\end{remark}

\begin{remark}
  \label{rem:conservative}
    According to \cite[Proposition 3.1]{CKW2021},
    lower bounds of HK$(\phi_j)$ or HK$_-(\phi_c,\phi_j)$ imply that 
    $(\Di,\calF)$ is conservative
    provided that the process has no killing inside $M$.
\end{remark}

\subsection{Subordinate process}
A subordinator is a non-negative L\'evy process.
By the L\'evy-Khintchin theorem, 
a subordinator $S_t$ is characterized by 
a constant $b \ge 0$ and a L\'evy measure $\nu(dt)$ as 
\begin{equation}
  \begin{split}
    E[e^{-\lambda S_t}] &= e^{-t \overline{\phi}(\lambda)},\\
    \overline{\phi}(\lambda) &= b \lambda + \int_{0}^{\infty} (1 - e^{-\lambda t}) \nu(dt).
  \end{split}
\end{equation}
$\overline{\phi}$ is called the Bernstein function.

Let $X_t$ be a $\mu$-symmetric Hunt process associated with a regular Dirichlet form $(\Di,\calF)$ 
and $S_t$ be an independent subordinator characterized by $(b,\nu)$.
By \^Okura \cite[Theorem 2.1]{Okura},
we have the expression of the Dirichlet form $(\Di^S,\calF^S)$ 
associated with the subordinate process $X_{S_t}$.
In particular, we assume that $X_t$ have a heat kernel $p(t,x,y)$,
$X_t$ is conservative and $b = 0$ (i.e. driftless subordinator),
then, the following theorem holds.
\begin{theorem}
  \label{thm:Okura}
  \begin{enumerate}[(i)]
    \item The Dirichlet form $(\Di^S,\calF^S)$ is regular.
    \item $\calF \subset \calF^S$.
    \item $(\Di^S,\calF^S)$ is pure-jump type and has the following explicit formula
  \begin{equation}
      \Di^S(f,g) = \int_{M \times M} (u(x)- u(y))(v(x) - v(y)) J(x,y) d\mu(x)d\mu(y),\quad f,g \in \calF^S,
  \end{equation}
  where the jump kernel $J(x,y)$ is 
  \begin{equation}
    J(x,y) = \frac{1}{2}\int_0^\infty p(t,x,y)d\nu(t).
\end{equation}
\end{enumerate}
\end{theorem}
The existence of the heat kernel and conservativeness 
are fulfilled 
for $(\Di,\calF)$ that satisfies HK$(\phi_j)$ or HK$_-(\phi_c,\phi_j)$
as mentioned in Remark \ref{rem:conservative}.

\section{Estimates of the subordinate jump kernel}\label{sec:EstimatesOfTheSubordinateJumpKernel}

In this section, we estimete jump kernel of subordinate processes 
whose original processes enjoy jump type or diffusion+jump type
heat kernel estimates.

For given scale functions $\phi$ and $\psi$, 
we always assume 
\begin{align}
  \int_0^\infty (1 \wedge t) \frac{dt}{t\psi(\phi^{-1}(t))} < \infty.
\end{align}
By the above integrability assumption, there exists a subordinator whose 
corresponding Bernstein function is 
\begin{align}
  \overline{\phi}(\lambda) &:= \int_0^{\infty} (1 - e^{-\lambda t}) \frac{dt}{t\psi(\phi^{-1}(t))}, \qquad \lambda > 0.
\end{align}
We will use this subordinator for our subordination.

\subsection{Pure jump type}

In this subsection, 
we will prove the following proposition.

\begin{proposition}
  \label{prop:jumpType}
  Assume that the regular Dirichlet form $(\Di,\calF)$ enjoys the jump type heat kernel estimate HK$(\phi_j)$ and 
  let $S$ be an independent subordinator whose Bernstein function is given in (\ref{eq:Bernstein function}).
  Then, the jump kernel $J(x,y)$ of the subordinate process satisfies 
  the following estimate: 
  \begin{align}
    J(x,y) &\simeq \frac{1}{V(x,r)} \overline{\phi}(\phi_j(r)^{-1}),
    \label{eq:jumpKernelEstimate}
  \end{align}
  where $r := d(x,y)$ and set $\phi := \phi_j$.
\end{proposition}

\begin{proof}

    We first prove the upper bound of (\ref{eq:jumpKernelEstimate}).
    By Theorem \ref{thm:Okura} and the heat kernel estimate (\ref{eq:JumpHKE}),
    the jump kernel is bounded by
    \begin{equation}
      \begin{split}
      J(x,y) &= \frac{1}{2}\int_0^\infty p(t,x,y) \frac{dt}{t\psi(\phi^{-1}(t))}\\ 
      &\le C \int_{0}^{\phi(r)}\frac{t}{V(x,r)\phi(r)}\frac{dt}{t\psi(\phi^{-1}(t))} 
       + C \int_{\phi(r)}^\infty \frac{1}{V(x,\phi^{-1}(t))}\frac{dt}{t\psi(\phi^{-1}(t))}  \\
       &=: \mathrm{(I)} + \mathrm{(II)}.
      \end{split}
    \end{equation}
    We first estimate (I).
    Since $1 - e^{-x} \ge xe^{-x}$ for $x \ge 0$,
    we get
    \begin{align}
      \overline{\phi}(\lambda) \ge e^{-1}\lambda \int_{0}^{\lambda^{-1}}\frac{dt}{\psi(\phi^{-1}(t))}.
    \end{align}
    By taking $\lambda := \phi(r)^{-1}$, we obtain the following inequality
    \begin{align}
      \mathrm{(I)} &\le \frac{C}{V(x,r)}\overline{\phi}(\phi(r)^{-1}).
      \label{eq:jumpTypeJumpKernelUpperBound1}
    \end{align}
    We next estimate (II).
    By the non-decreasing property of $V$, we have   
    \begin{align}
      \mathrm{(II)} &\le \frac{C}{V(x,r)}\int_{\phi(r)}^{\infty}\frac{dt}{t\psi(\phi^{-1}(t))}.
    \end{align}
    Hence by taking $\lambda := \phi(r)^{-1}$,
    \begin{equation}
      \overline{\phi}(\phi(r)^{-1}) 
      \ge \int_{\phi(r)}^{\infty}(1 - e^{-\phi(r)^{-1}t})\frac{dt}{t\psi(\phi^{-1}(t))}
      \ge (1 - e^{-1})\int_{\phi(r)}^{\infty}\frac{dt}{t\psi(\phi^{-1}(t))},
    \end{equation}
    hence we obtain 
    \begin{align}
      \mathrm{(II)} &\le \frac{C}{V(x,r)}\overline{\phi}(\phi(r)^{-1}).
      \label{eq:jumpTypeJumpKernelUpperBound2}
    \end{align}
    By (\ref{eq:jumpTypeJumpKernelUpperBound1}) and (\ref{eq:jumpTypeJumpKernelUpperBound2}),
    we obtain the upper bound of the jump kernel
    \begin{align}
      J(x,y) &\le \frac{C}{V(x,r)}\overline{\phi}(\phi(r)^{-1}).
    \end{align}

    We next prove the lower bound of (\ref{eq:jumpKernelEstimate}).
    Put $\mu(t) := 1/\psi(\phi^{-1}(t))$.
    For all $\lambda > 0$, the Bernstein function $\overline{\phi}$ is estimated as follows,
    \begin{equation}
      \begin{split}
        \overline{\phi}(\lambda) &= \int_{0}^{\infty} (1 - e^{-\lambda t})\frac{1}{t}\mu(t)dt \\
      &= \int_{0}^{\infty} \lambda \int_{0}^{t} e^{-\lambda s}ds \frac{1}{t} \mu(t) dt\\
      &= \lambda \int_{0}^{\infty} \int_{s}^{\infty} e^{-\lambda s} \frac{1}{t} \frac{\mu(t)}{\mu(s)}\mu(s)dtds\\
      &\le C \lambda \int_{0}^{\infty} e^{-\lambda s}\mu(s) \int_{s}^{\infty} \left(\frac{s}{t}\right)^{\gamma_1/\beta_2} \frac{1}{t} dtds\\
      &= C\frac{\beta_2}{\gamma_1} \lambda \int_{0}^{\infty} e^{-\lambda s}\mu(s) ds\\
      &\le C \lambda \int_{0}^{\phi(r)} \mu(s) ds + C \lambda\int_{\phi(r)}^{\infty}e^{-\lambda s}\mu(s) ds
      =: \mathrm{(I)} + \mathrm{(II)}.
      \end{split}
      \label{eq:UpperEstimateOfBernsteinFunction}
    \end{equation}
    Hereinafter, we take $\lambda := \phi(r)^{-1}$.

    For (I), by definition, we get 
    \begin{equation}
      \frac{1}{V(x,r)}\mathrm{(I)}
      = C \int_{0}^{\phi(r)}\frac{t}{V(x,r)\phi(r)} \frac{1}{t\psi(\phi^{-1}(t))}dt
      \le C J(x,y).
    \end{equation}
    For (II) using (VD), we have 
    \begin{equation}
      \begin{split}
      \frac{1}{V(x,r)}\mathrm{(II)} 
      &= C \int_{\phi(r)}^{\infty}\frac{V(x,\phi^{-1}(t))}{V(x,r)} \frac{\lambda t e^{-\lambda t}}{V(x,\phi^{-1}(t))} \frac{1}{t}\mu(t)dt \\
      &\le C \int_{\phi(r)}^{\infty} e^{-\lambda t } \left(\lambda t \right)^{d_2/\beta_1 + 1} \frac{1}{V(x,\phi^{-1}(t))} \frac{1}{t}\mu(t)dt\\
      &\le C \int_{\phi(r)}^{\infty} \frac{1}{V(x,\phi^{-1}(t))} \frac{1}{t\psi(\phi^{-1}(t))} dt\\
      &\le C J(x,y),
      \label{eq:UpperEstimateOfBernsteinFunction2}
    \end{split}
  \end{equation}
    where the second inequality comes from 
    \begin{align*}
      \sup_{t>0}e^{-t}t^{d_2/\beta_1+1} &= \left(\frac{d_2}{\beta_1} + 1\right)^{d_2/\beta_1 + 1}e^{-(d_2/\beta_1 + 1)}.
    \end{align*}
    By (\ref{eq:UpperEstimateOfBernsteinFunction})-(\ref{eq:UpperEstimateOfBernsteinFunction2}),
    we obtain the desired lower bound
    \begin{align}
      \frac{C}{V(x,r)}\overline{\phi}(\phi(r)^{-1}) &\le J(x,y).
    \end{align}
\end{proof}

\begin{remark}
  \begin{enumerate}[(i)]
    \item  We do not use (RVD) to prove Proposition \ref{prop:jumpType}.
    \item We remark that estimates of jump kernels using Bernstein functions $\overline{\phi}$ 
      already exist, for example \cite[p.70]{CKW2021} \cite[Proposition 4.2]{KM2012}.
  \end{enumerate}
\end{remark}

\subsection{Diffusion+jump type}

In this subsection, 
we will prove the following proposition.

\begin{proposition}
  \label{prop:diffusion+jumpType}
  Assume $(\Di,\calF)$ satisfies the diffusion+jump type heat kernel estimate HK$_-(\phi_c,\phi_j)$ and 
  let $S$~ be an independent subordinator whose Bernstein function is given in (\ref{eq:Bernstein function})
  with $\phi := \phi_c \wedge \phi_j$.
  Then, the jump kernel $J(x,y)$ of the subordinate process satisfies 
  the following estimate: 
  \begin{equation}
    J(x,y) \simeq \frac{1}{V(x,r)}\left(\frac{1}{\psi(r)} +\overline{\phi}(\phi_j(r)^{-1})\right),
  \end{equation}
  where $r := d(x,y)$.
\end{proposition}

We note that the diffusion version of Proposition \ref{prop:diffusion+jumpType} 
is known.
\begin{proposition}(\cite[Lemma 4.2]{BKKL})
  \label{prop:diffusionType}
  Assume $(\Di,\calF)$ satisfies the diffusion type heat kernel estimate for a scale function $\phi_c$ and 
  let $S$~ be an independent subordinator whose Bernstein function is given in (\ref{eq:Bernstein function})
  with $\phi := \phi_c$.
  Then, the jump kernel $J(x,y)$ of the subordinate process satisfies 
  the following estimate: 
  \begin{equation}
    J(x,y) \simeq \frac{1}{V(x,r)} \frac{1}{\psi(r)},
  \end{equation}
  where $r := d(x,y)$.
\end{proposition}

We divide the proof of Proposition \ref{prop:diffusion+jumpType} into two lemmas.
\begin{lemma}[Upper bound]
  \label{lem:diffusion+jumpTypeUpperBound}
  The following estimate holds under the assumption in Proposition \ref{prop:diffusion+jumpType},
  \begin{equation}
    J(x,y) \le \frac{C}{V(x,r)}\left(\frac{1}{\psi(r)} +\overline{\phi}(\phi_j(r)^{-1})\right), \qquad r > 0.
  \end{equation}
\end{lemma}

\begin{proof}
  Since the scale function $\phi(r)$ and the heat kernel 
  behave differently between small distance and large distance,
  we separate two cases.
  \begin{description}
  \item[Case 1 ($r \ge 1$).] 

  Since the heat kernel $p(t,x,y)$ is bounded by
  \begin{align*}
    p(t,x,y) &\le c_3 \left(p^{(c)}(c_4t,x,y) + \frac{1}{V(x,\phi^{-1}(t))}\wedge p^{(j)}(t,x,y)\right),
  \end{align*}
  we have a bound
  \begin{equation}
    \begin{split}
    J(x,y) &\le c_3 \int_0^\infty p^{(c)}(c_4t,x,y) \frac{dt}{t\psi(\phi^{-1}(t))} + c_3 \int_0^\infty \left( \frac{1}{V(x,\phi^{-1}(t))}\wedge p^{(j)}(t,x,y)\right)\frac{dt}{t\psi(\phi^{-1}(t))} \\
    &= \mathrm{(I)} + \mathrm{(II)}.
    \end{split}
    \label{eq:JumpKernelBoundInLemma1}
  \end{equation}
  For (I), we have
  \begin{align}
    \mathrm{(I)} 
    &\le c_3 \int_0^\infty p^{(c)}(c_4t,x,y) \frac{dt}{t\psi(\phi_c^{-1}(t))}
    \le \frac{C}{V(x,r)\psi(r)},
  \end{align}
  where the last inequality follows from Proposition \ref{prop:diffusionType}.
  (II) is bounded by
  \begin{equation}
    \begin{split}
      \mathrm{(II)} 
      &= c_3 \int_0^\infty \left(\frac{1}{V(x,\phi^{-1}(t))}\wedge\frac{t}{V(x,r)\phi(r)}\right)\frac{dt}{t\psi(\phi^{-1}(t))}\\
      &\le C\frac{\overline{\phi}(\phi(r)^{-1})}{V(x,r)} = C\frac{\overline{\phi}(\phi_j(r)^{-1})}{V(x,r)},
    \end{split}
  \end{equation}
  where the inequality follows from jump type estimate in Proposition \ref{prop:jumpType}.

  \item[Case 2 ($r < 1$).] 
By \cite[Remark 1.12]{CKW2020}, the heat kernel $p(t,x,y)$ is bounded by
\begin{equation}
  p(t,x,y) \le \begin{cases}
    C\left(\cfrac{t}{V(x,r)\phi_j(r)} + p^{(c)}(c_2t,x,y)\right), & 0 < t \le \phi_c(c_1 r)\\
    \cfrac{C}{V(x,\phi^{-1}(t))}, & \phi_c(c_1 r) < t
  \end{cases}.
\end{equation}
Hence, the jump kernel is bounded by
\begin{align}
  J(x,y) 
  &\le C\int_0^{\phi_c(c_1 r)} \frac{t}{V(x,r)\phi_j(r)}\frac{dt}{t\psi(\phi^{-1}(t))}
  + C\int_{0}^{\phi_c(c_1r)}p^{(c)}(c_2t,x,y)\frac{dt}{t\psi(\phi^{-1}(t))}\nonumber\\
  &\quad + C\int_{\phi_c(c_1 r)}^{\infty} \frac{1}{V(x,\phi^{-1}(t))}\frac{dt}{t\psi(\phi^{-1}(t))}\nonumber\\
  &= \mathrm{(I)} + \mathrm{(II)} + \mathrm{(III)}.
  \label{eq:diffusion+jumpJumpkernelUpperBoundAtProof}
\end{align}
Define a truncated Bernstein function $\phi^W(\lambda;r)$ by
\begin{align}
  \label{eq:truncatedBernsteinFunction}
  \phi^W(\lambda;r) &:= \int_{0}^{\phi(r)}(1-e^{-\lambda t})\frac{dt}{t\psi(\phi^{-1}(t))}.
\end{align}
Since $1 - e^{-x} \ge x e^{-x}, x \ge 0$, by taking $\lambda := \phi_j(r)^{-1}$ we have 
\begin{align}
  \phi^W(\lambda;r) 
  &\ge \int_0^{\phi_c(r)}\lambda t e^{-\lambda t}\frac{dt}{t\psi(\phi^{-1}(t))}
  \ge e^{-1}\lambda\int_0^{\phi_c(r)} \frac{dt}{\psi(\phi^{-1}(t))}, \qquad \lambda > 0,
\end{align}
here we use $\phi_c(r) \le \phi_j(r)$ for $r \le 1$.
Hence, the first term of (\ref{eq:diffusion+jumpJumpkernelUpperBoundAtProof}) is bounded by
\begin{equation}
  \mathrm{(I)} 
  \le \frac{C}{V(x,r)\phi_j(r)}\int_0^{\phi_c(c_1 r)} \frac{dt}{\psi(\phi^{-1}(t))}
  \le \frac{C}{V(x,r)}\phi^W(\phi_j(r)^{-1};r)
  \le \frac{C}{V(x,r)} \overline{\phi}(\phi_j(r)^{-1}).
  \label{eq:diffusion+jumpTypeJumpKernelUpperBound1}
\end{equation}
The second term of (\ref{eq:diffusion+jumpJumpkernelUpperBoundAtProof}) is same as Case 1 $(r \ge 1)$ and we have 
\begin{equation}
  \mathrm{(II)} \le C \int_{0}^{\infty} p^{(c)}(c_2t,x,y)\frac{dt}{t\psi(\phi_c^{-1}(t))}  
  \le \frac{C}{V(x,r)\psi(r)}.
  \label{eq:diffusion+jumpTypeJumpKernelUpperBound2}
\end{equation}
Finally, putting $c := c_1 \wedge 1$, the third term of (\ref{eq:diffusion+jumpJumpkernelUpperBoundAtProof}) is bounded by
\begin{equation}
  \begin{split}
  \mathrm{(III)} 
  &\le \frac{C}{V(x,r)\psi(r)}\frac{V(x,r)}{V(x,cr)}\int_{\phi(c r)}^\infty \frac{\psi(r)}{\psi(\phi^{-1}(t))}\frac{V(x,cr)}{V(x,\phi^{-1}(t))}\frac{1}{t}dt \\
  &\le \frac{C}{V(x,r)\psi(r)}\int_{\phi(r)}^\infty \frac{\psi(r)}{\psi(\phi^{-1}(t))}\frac{1}{t}dt\\
  &\le \frac{C}{V(x,r)\psi(r)}\int_{\phi(r)}^\infty \left(\frac{\phi(r)}{t}\right)^{\gamma_1/\beta_2}\frac{1}{t}dt \\
  &\le \frac{C}{V(x,r)\psi(r)}.
\end{split}
  \label{eq:diffusion+jumpTypeJumpKernelUpperBound3}
\end{equation}
where second inequality comes from non-decreasing property of $V(x,r)$ 
and (VD).
Summing up the 
(\ref{eq:diffusion+jumpTypeJumpKernelUpperBound1}), 
(\ref{eq:diffusion+jumpTypeJumpKernelUpperBound2}) and
(\ref{eq:diffusion+jumpTypeJumpKernelUpperBound3}), we obtain desired result.
\end{description}
\end{proof}

\begin{lemma}[Lower bound]
  \label{lem:diffusion+jumpTypeLowerBound}
  The following estimate holds under the assumption in Proposition \ref{prop:diffusion+jumpType},
  \begin{equation}
    \frac{C}{V(x,r)}\left(\frac{1}{\psi(r)} +\overline{\phi}(\phi_j(r)^{-1})\right) \le J(x,y), \qquad r > 0.
  \end{equation}
\end{lemma}

\begin{proof}
  As in the proof of Lemma \ref{lem:diffusion+jumpTypeUpperBound}, 
  we divide the proof into two cases.
  \begin{description}
    \item[Case 1 ($r \ge 1$).] 
    
    By HK$_{-}(\phi_c,\phi_j)$, the heat kernel is bounded below by
    \begin{equation}
      p(t,x,y) \ge 
      \begin{cases}
        \cfrac{Ct}{V(x,r)\phi_j(r)}, & t < \phi(c_1 r)\\
        \cfrac{C}{V(x,\phi^{-1}(t))}, & \phi(c_1 r) \le t 
      \end{cases}.
      \label{eq:HK-}
    \end{equation}
  Using (\ref{eq:UpperEstimateOfBernsteinFunction}), we have 
  \begin{equation}
    \overline{\phi}(\lambda)  
    \le C \lambda \int_0^{\phi(c_1r)}\frac{1}{\psi(\phi^{-1}(t))}dt
    +  C \lambda \int_{\phi(c_1 r)}^\infty e^{-\lambda t}\frac{1}{\psi(\phi^{-1}(t))}dt
    = \mathrm{(I)} + \mathrm{(II)}.
    \label{eq:UpperEstimateOfBernsteinFunctionAtProof}
  \end{equation}
  Let $\lambda = \phi_j(r)^{-1}$.
  The first term of (\ref{eq:UpperEstimateOfBernsteinFunctionAtProof}) is bounded by
  \begin{equation}
    \begin{split}
      \frac{1}{V(x,r)}\mathrm{(I)} 
      &= C\int_0^{\phi(c_1 r)} \frac{t}{V(x,r)\phi_j(r)} \frac{1}{t\psi(\phi^{-1}(t))}dt\\
    &\le C \int_0^{\phi(c_1 r)} p(t,x,y)\frac{1}{t\psi(\phi^{-1}(t))}dt.
    \end{split}
    \label{eq:UpperEstimateOfBernsteinFunctionAtProof1}
  \end{equation}
  The second term of (\ref{eq:UpperEstimateOfBernsteinFunctionAtProof}) is bounded by
  \begin{equation}
    \begin{split}
      \frac{1}{V(x,r)}\mathrm{(II)} 
      &= C  \int_{\phi(c_1r)}^\infty \lambda t e^{-\lambda t} \frac{V(x,\phi^{-1}(t))}{V(x,r)} \frac{1}{V(x,\phi^{-1}(t))}\frac{1}{t\psi(\phi^{-1}(t))}dt\\
      &\le  C  \int_{\phi(c_1r)}^\infty \frac{1}{V(x,\phi^{-1}(t))}\frac{1}{t\psi(\phi^{-1}(t))}dt\\
    &\le  C \int_{\phi(c_1r)}^\infty p(t,x,y)\frac{1}{t\psi(\phi^{-1}(t))}dt.
    \end{split}
    \label{eq:UpperEstimateOfBernsteinFunctionAtProof2}
  \end{equation}
  Here the first inequality is from 
  \begin{equation}
    \begin{split}
      \lambda t e^{-\lambda t}\frac{V(x,\phi^{-1}(t))}{V(x,r)}
      &\le \lambda t e^{-\lambda t} \left(1_{\{t \le \phi_j(r)\}} + C\left(\frac{t}{\phi_j(r)}\right)^{d_2/\beta_1}1_{\{t > \phi_j(r)\}}\right)\\
      &\le \sup_{x\ge 0}\ x(1 + x^{d_2/\beta_1})e^{-x}, \qquad t > 0,
    \end{split}
  \end{equation}  
  where 
  we use $V(x,\phi^{-1}(r)) \le V(x,r)$ for $t \le \phi(r) = \phi_j(r)$
  and (VD) for $t > \phi_j(r)$.
  Hence, summing up (\ref{eq:UpperEstimateOfBernsteinFunctionAtProof1}) 
  and (\ref{eq:UpperEstimateOfBernsteinFunctionAtProof2}), 
  we have 
  \begin{equation}
    \frac{\overline{\phi}(\phi_j(r)^{-1})}{V(x,r)} 
    \le C \int_{0}^{\infty} p(t,x,y) \frac{dt}{t\psi(\phi^{-1}(t))}
    \le CJ(x,y).
  \label{eq:diffusion+jumpTypeJumpKernelLowerBound1InCase1}
\end{equation}
  Finally, let $c := c_1 \vee 1$ which implies $\phi(cr) = \phi_j(cr)$.
  By Theorem \ref{thm:Okura} and the heat kernel estimate HK$_{-}(\phi_c,\phi_j)$,
  we have 
  \begin{equation}
    \begin{split}
    J(x,y) 
    &\ge C\int_{\phi(c r)}^{2 \phi (c r)} \frac{1}{V(x,\phi^{-1}(t))} \frac{1}{t\psi(\phi^{-1}(t))} dt\\
    &= \frac{C}{V(x,r)\psi(cr)} \frac{V(x,r)}{V(x,c r)} \int_{\phi(cr)}^{2\phi( c r)} \frac{V(x,c r)}{V(x,\phi^{-1}(t))} \frac{\psi(cr)}{t\psi(\phi^{-1}(t))} dt\\
    &\ge \frac{C}{V(x,r)\psi(r)}\int_{\phi(c r)}^{2\phi(c r)} \left(\frac{\phi(c r)}{t}\right)^{d_2/\beta_1 + \gamma_2/\beta_1} \frac{1}{t}dt\\
    &\ge \left(1 - 2^{-(d_2/\beta_1 + \gamma_2/\beta_1)}\right)\frac{C}{V(x,r)\psi(r)}.
  \end{split}
  \label{eq:diffusion+jumpTypeJumpKernelLowerBound2InCase1}
\end{equation}
  By (\ref{eq:diffusion+jumpTypeJumpKernelLowerBound1InCase1})
  and (\ref{eq:diffusion+jumpTypeJumpKernelLowerBound2InCase1}),
   we obtain the desired result.

  \item[Case 2 ($r < 1$).] 
  By HK$_{-}(\phi_c,\phi_j)$, the heat kernel is bounded below by (\ref{eq:HK-}) again.
  Since $1 - e^{-x} \le x$ for $x \ge 0$ and $\phi(r) = \phi_c(r) \wedge \phi_j(r)$,
  we have an estimate 
  for the truncated Bernstetin function,
  \begin{equation}
    \phi^W(\lambda;r) = \int_{0}^{\phi(r)}(1- e^{-\lambda t}) \frac{dt}{t\psi(\phi^{-1}(t))} 
    \le \lambda \int_0^{\phi(r)} \frac{dt}{\psi(\phi^{-1}(t))}.
  \end{equation}
  Therefore, taking $\lambda = \phi_j(r)^{-1}$, we obtain a part of the lower bound 
  \begin{equation}
    \begin{split}
    \frac{1}{V(x,r)}\phi^W (\phi_j(r)^{-1};r)
    &\le C \int_0^{\phi(r)} \frac{t}{V(x,r)\phi_j(r)} \frac{dt}{t\psi(\phi^{-1}(t))} \\
    &\le C \int_0^{\phi(c_1 r)} p(t,x,y) \frac{dt}{t\psi(\phi^{-1}(t))} \le CJ(x,y),
    \end{split}
    \label{eq:diffusion+jumpTypeJumpKernelLowerBound1}
  \end{equation}
  by using the doubling properly of $\phi$ and $\psi$.
  Moreover, by the same reason as in Case 1, we have 
  \begin{equation}
    \begin{split}
    J(x,y) &\ge \frac{C}{V(x,r)\psi(r)}.
  \end{split}
  \label{eq:diffusion+jumpTypeJumpKernelLowerBound2}
\end{equation}
  In order to get the desired lower bound,
   we claim that there exists a constant $C > 0$ such that  
  \begin{equation}
    \overline{\phi}(\phi_j(r)^{-1}) \le \phi^W (\phi_j(r)^{-1};r) + C \frac{1}{\psi(r)}, \quad r \le 1,
    \label{eq:ClaimCompareToTruncatedBernsteinFunction}
  \end{equation}
  holds.
  Indeed, by definition of $\phi^W$ we have 
  \begin{equation}
    \begin{split}
        \overline{\phi}(\phi_j(r)^{-1}) 
        &= \phi^W (\phi_j(r)^{-1};r) 
         + \int_{\phi_c(r)}^{\phi_j(r)} (1 - e^{-\phi_j(r)^{-1} t })\frac{dt}{t\psi(\phi^{-1}(t))}\\
        &\quad + \int_{\phi_j(r)}^\infty  (1 - e^{-\phi_j(r)^{-1} t })\frac{dt}{t\psi(\phi^{-1}(t))}.
    \end{split}
    \label{eq:CompareToTruncatedBernsteinFunction}
  \end{equation}
  The second term of (\ref{eq:CompareToTruncatedBernsteinFunction}) is bounded by 
  \begin{equation}
    \begin{split}
    \int_{\phi_c(r)}^{\phi_j(r)} (1 - e^{-\phi_j(r)^{-1} t })\frac{dt}{t\psi(\phi^{-1}(t))}
    &\le \frac{1}{\phi_j(r)}\int_{\phi_c(r)}^{\phi_j(r)} \frac{dt}{\psi(\phi^{-1}(t))}\\
    &\le \frac{1}{\phi_j(r)}\frac{1}{\psi(r)}\int_{\phi_c(r)}^{\phi_j(r)} dt\\
    &\le \frac{1}{\psi(r)}.
    \end{split}
  \end{equation}
  The third term of (\ref{eq:CompareToTruncatedBernsteinFunction}) is bounded by
  \begin{equation}
    \begin{split}
          \int_{\phi_j(r)}^\infty  (1 - e^{-\phi_j(r)^{-1} t })\frac{dt}{t\psi(\phi^{-1}(t))}
    &\le  \frac{1}{\psi(r)} \int_{\phi_j(r)}^\infty \frac{\psi(r)}{t\psi(\phi^{-1}(t))}dt\\
    &\le \frac{C}{\psi(r)} \int_{\phi_j(r)}^\infty \left(\frac{\phi(r)}{t}\right)^{\gamma_1/\beta_2} \frac{1}{t}dt\\
    &\le \frac{C}{\psi(r)}.
    \end{split}
  \end{equation}
  Hence the claim (\ref{eq:ClaimCompareToTruncatedBernsteinFunction}) is proved.
  Summing up 
  (\ref{eq:diffusion+jumpTypeJumpKernelLowerBound1}), 
  (\ref{eq:diffusion+jumpTypeJumpKernelLowerBound2}) and
  (\ref{eq:ClaimCompareToTruncatedBernsteinFunction}),
  we obtain the desired lower bound.
  \end{description}
\end{proof}

  In Proposition \ref{prop:diffusionType} \cite[Lemma 4.2]{BKKL}, 
  the jump kernel of the Dirichlet form of the subordinate process is 
  estimated by 
  \begin{equation}
    J(x,y) \simeq \frac{1}{V(x,r)\psi(r)},
    \label{eq:jumpKernelEstimateOfDiffusionType}
  \end{equation}
  for the same scale function and the same subordinator
  when the heat kernel of the original process 
  satisfies sub-Gaussian heat kernel estimates.
  It is natural to ask 
  whether (\ref{eq:jumpKernelEstimate}), (\ref{eq:jumpKernelEstimateOfDiffusionType})
  and (\ref{eq:jumpKernelEstimateOfdiffusion+jump})
  are comparable.

  In general, resulting jump kernels of the jump type $\overline{\phi}(\phi_j(r)^{-1})$ 
  and of the  diffusion type $1/\psi(r)$ are not comparable
  (see \cite[Example 1.1 and Lemma 6.1]{CKW2022}).
  For readers' convenience, we discuss this for some concrete case below.

  First, note that 
  the following estimate holds under the assumption in Proposition \ref{prop:diffusion+jumpType},
    \begin{equation}
      \begin{split}        
        \frac{1}{\psi(r)} &\le \overline{\phi}(\phi_j(r)^{-1}),\quad  r \ge 1,
      \end{split}
    \end{equation}
    see \cite[Lemma 4.3]{BKKL}.
\begin{example}
  \label{ex:comparable}
   Let $\beta \ge 2, \gamma_1 \in (0,1), \gamma_2 \in (1,\infty)$,
   and take $\phi(r) := r^\beta, \psi(r) := r^{\gamma_1 \beta} \vee r^{\gamma_2\beta}$.
   Then, 
   \begin{equation}
    \overline{\phi}(\phi_j(r)^{-1}) \simeq r^{-\gamma_1 \beta} \wedge r^{-\beta}
   \end{equation}
   and 
   \begin{equation}
    \frac{1}{\psi(r)} \simeq r^{-\gamma_1 \beta} \wedge r^{-\gamma_2 \beta}.
   \end{equation}
   Clearly, these are not comparable in general.
\end{example}
  A sufficient condition for the comparablilty is   
  there exists $0 < a < 1$ such that  the following holds
  \begin{align}
  \frac{\psi(\phi^{-1}(T))}{\psi(\phi^{-1}(t))} &\le C \left(\frac{T}{t}\right)^a,\quad 0 < t \le T.
  \end{align}

\subsection{Proof of Theorem \ref{thm:diffusion+jump} and the pure jump case}

We will prove Theorem \ref{thm:diffusion+jump}.
The core of the proof is to estimate the jump kernel of the subordinate Dirichlet form
in Proposition \ref{prop:diffusion+jumpType}.
For the proof, we use the following lemma.
\begin{lemma}(\cite[Lemma 3.5]{LM})
  \label{lem:LMLemma3.5}
  Let $\phi,\psi$ be two scale functions.
  Then 
  $\int_{0}^{1}\frac{\phi(s)}{s\psi(s)}ds < \infty$ holds 
  if and only if $\int_{0}^{1}\frac{ds}{\psi(\phi^{-1}(s))} < \infty$ holds.
\end{lemma}
\begin{proof}
  Put $\Phi(t) := \int_0^t \frac{\phi(r)}{r}dr$.
  Then $\Phi \simeq \phi$ and $\frac{d\Phi}{dr}\simeq \frac{\phi(r)}{r}$.
  Hence we conclude the equivalence by the change of variables.
\end{proof}

\begin{proof}[Proof of Theorem \ref{thm:diffusion+jump}]
  (b) implies (a) is obvious.
  (a) implies $\int_{(0,1)} \frac{dt}{\psi(\phi^{-1}(t))} < \infty$.
  By Lemma \ref{lem:LMLemma3.5} this is equivalent to (c).
  Assume (c). By Lemma \ref{lem:LMLemma3.5}, we can take a subordinator of (\ref{eq:Bernstein function}).
  Then, Proposition \ref{prop:diffusion+jumpType} implies (b).
\end{proof}

If we assume that the 
original Dirichlet form satisfies the pure jump type heat kernel estimates HK$(\phi_j)$ 
instead of the diffusion+jump type heat kernel estimates HK$_-(\phi_c,\phi_j)$,
the same result holds.
The following corollary is a variant of Theorem \ref{thm:diffusion+jump}

\begin{corollary}
  \label{cor:jump}
  Let $(\Di,\calF)$ be a regular non-killing Dirichlet form on $L^2(M;\mu)$ that 
  satisfies the jump type heat kernel estimates HK$(\phi_j)$.
  $X$ denotes the $\mu$-symmetric Hunt process corresponding to $(\Di,\calF)$
  and set $\phi := \phi_j$.
  Given a scale function $\psi$, the following are equivalent. 
  \begin{enumerate}[(a)]
    \item There exists a pure jump regular Dirichlet form $(\Di^j,\calF^j)$  whose jump kernel satisfies (\ref{eq:jumpKernelEstimate}).
    \item There exists a subordinator $S_t$\ such that the jump kernel of the subordinate process $X_{S_t}$ satisfies  (\ref{eq:jumpKernelEstimate}).
    \item The scale function $\psi$ satisfies
    \begin{equation}
      \int_0^1 \frac{\phi(s)}{s\psi(s)}ds < \infty.
    \end{equation}
  \end{enumerate}
\end{corollary}

\begin{proof}
  The proof is the same as that of 
  Theorem \ref{thm:diffusion+jump}.
\end{proof}

\bigskip
\noindent \textbf{Acknowledgements.}\ 
The author would like to 
express his deep gratitude to Prof. Takashi Kumagai for
discussions and feedbacks on the whole paper.
The author thanks 
Prof. Naotaka Kajino,
Prof. Mathav Murugan,
Dr. Ryosuke Shimizu
and Prof. Yuichi Shiozawa
for helpful comments on the first draft.

\smallskip
\noindent \textbf{Ryuto Kushida:}\ 

\smallskip \noindent
Department of Pure and Applied Mathematics, Graduate School of Fundamental Science and Engineering, Waseda University,
3-4-1 Okubo, Shinjuku, Tokyo 169-8555, Japan.

\noindent E-mail: \texttt{otuyrdiades@ruri.waseda.jp}

\end{document}